\documentclass[ieot]{birkjour-arxive}
\usepackage{amssymb}
\usepackage{amsmath}
\usepackage{graphicx,color}
\usepackage{latexsym}
\usepackage{amsfonts}
\usepackage{accents}
\usepackage{psfrag}

\usepackage{extarrows}

\newtheorem{theorem}{Theorem}[section]
\newtheorem{lemma}[theorem]{Lemma}

\theoremstyle{definition}

\newtheorem{remark}[theorem]{Remark}

\numberwithin{equation}{section}

\DeclareMathOperator\Real{Re}
\DeclareMathOperator\Imag{Im}
\renewcommand\Re{\Real}
\renewcommand\Im{\Imag}

\newcommand{\boldN}{{\bf N}}
\newcommand{\boldS}{{\bf S}}

\newcommand{\al}{\alpha}
\newcommand{\la}{\lambda}
\newcommand{\be}{\beta}
\newcommand{\de}{\delta}
\newcommand{\ga}{\gamma}
\newcommand{\si}{\sigma}

\newcommand{\calG}{{\mathcal G}}

\newcommand{\calS}{{\mathcal S}}

\newcommand{\wh}{\widehat}
\newcommand{\wt}{\widetilde}

\newcommand{\dd}{{\rm d}}

\newcommand{\CC}{\mathbb C}

\newcommand{\RR}{\mathbb R}
\newcommand{\ZZ}{\mathbb Z}

\newcommand{\ii}{{\text{\rm i}}}
\newcommand{\ee}{{\text{\rm e}}}
\newcommand{\ov}{\overline}

\newlength{\dhatheight}

\newcommand{\void}[1]{}

\newcounter{counter_a}

\begin{document}
%\begin{document}

\title[Transfer functions and local spectral uniqueness]{Transfer functions and local spectral uniqueness for Sturm-Liouville operators,\\ canonical systems and strings}

\author[H.~Langer]{Heinz Langer}
\address{Institute  of Analysis and Scientific Computing,\\
Vienna University of Technology,\\
1040 Vienna, Austria}
\email{heinz.langer@tuwien.ac.at}

%\thanks{}

\dedicatory{Dedicated to Zolt\'an Sasv\'ari on the occasion of his 60th birthday}

\begin{abstract}
It is shown that transfer functions, which play a crucial role in M.\,G.\ Krein's study of inverse spectral problems, are a proper tool
to formulate local spectral uniqueness conditions.
%
%Tychonoff/Samarski: Dgln. der Math. Physik, 35-01-T\\
%Wladimiroff: Gleichungen der Math. Physik, 35-01-V

\end{abstract}



\maketitle{}

\today
\section{Introduction}
The main object of study in this note is a Sturm-Liouville operator $-y''+qy$ on an interval $[0,\ell)$ and with a self-adjoint boundary condition at $x=0$. By the Borg-Mar\v cenko theorem,  a spectral measure of this operator determines the potential and the boundary condition uniquely.
In 1998 B.~Simon proved a local version of the Borg-Mar\v cenko theorem (see \cite[Theorem 1.2]{Si1}): Given two such Sturm-Liouville problems with potentials $q_1$ and $q_2$ on intervals $[0,\ell_1),\, [0,\ell_2)$, he formulates a condition for the coincidence of the potentials on some sub-interval
$[0,a],\,0<a\le\min\{\ell_1,\ell_2\}$ and of the boundary conditions at $0$. B.~Simon's  condition is formulated in terms of the Weyl-Titchmarsh function  of the problem which is the Stieltjes transform of the spectral measure.
In \cite{B}  C.\ Bennewitz  gave a  very short  proof of B.~Simon's result. 

F.~Gesztesy and B.~Simon write in \cite[p. 274]{GS} that  `it took  over 45 years to improve Theorem 2.1' (the Borg-Mar\v cenko theorem) `and derive its  local counterpart'. It is the aim of this note to show that the method and results of M.G.~Krein on 
 direct and inverse spectral problems for  Sturm-Liouville operators in the beginning of the 1950s (see, e.g., \cite{K52}) provide a simple criterion for local uniqueness of the potential in terms of Krein's transfer function. From this criterion B.~Simon's uniqueness result  follows by  a Phragm$\acute{\rm e}$n-Lindel\"of argument. We remark that the amplitude function in \cite{Si1} is essentially the derivative of\, Krein's transfer function;  the class of spectral measures in \cite{K52} and in the present note is larger than that in \cite{GS}.

The connection between the Sturm-Louville operator and its transfer function corresponds to the  connection between a symmetric Jacobi matrix and the Hamburger moment problem; see also Subsection 5.3.

As for Sturm-Liouville problems, transfer functions can be defined for canonical system and strings, see e.g.\ \cite{KL5},  and these can also be used to give criteria for local spectral uniqueness. We formulate some of these results in Section~4~below.

The transfer functions are continuous and have the property, that a certain hermitian kernel is positive definite. This fact yields  integral representations of the transfer functions with respect to measures which are the spectral measures of the differential operator. The problem to determine all spectral measures of a symmetric but not self-adjoint differential operator is therefore closely related to the problem of extending a function, given on some interval and for which a certain kernel is positive definite, to a maximal interval such that this kernel is still positive definite (comp.  \cite{KL5}).

I thank Dr.\ Sabine Boegli from the Institute of Mathematics of the University of Bern for the calculations and plots of the examples, and her as well as V.N.~Pivovarcik, G.~Freiling and V.~Yurko for valuable remarks.

\section{Spectral measures and transfer functions for Sturm-Liouville problems}

\noindent{2.1.} Consider   the Sturm-Liouville problem 
\begin{equation}\label{prob}
-y''(x)+q(x)y(x)-zy(x)=0,\ x\in [0,\ell),\ z\in\CC,\quad  y'(0)-   h\, y(0)=0,
\end{equation}
where $0<\ell\le\infty,\, q\in L^1_{\it loc}([0,\ell)), h\in\RR$; the case $h=\infty$, that is the boundary condition $y(0)=0$,  is not considered in this note. We set $h=\cot\al$ with $0<\al<\pi$, and denote by $\varphi(x;z), \psi(x;z)$ the solutions of the differential equation in \eqref{prob} satisfying the initial conditions
\[
 \varphi(0;z)=\sin\al,\ \varphi'(0;z)=\cos\al,\quad \psi(0;z)=-\cos\al,\ \psi'(0;z)=\sin\al.
\]
Hence $\varphi(\cdot;z)$ is the solution of boundary value problem \eqref{prob}.

We  recall the definition of a  spectral measure of the problem \eqref{prob}. Denote by $\mathcal L_0$  the set of all functions $f\in L^2(0,\ell)$ which vanish identically near $\ell$.
The {\it Fourier transformation} $\mathcal F$ of the problem \eqref{prob} is given by
 \[
\mathcal F(y;\la):=\int_0^\ell y(x)\varphi(x;\la)\dd x,\quad \la\in\RR,\ y\in\mathcal L_0.
 \] 
Clearly, $\mathcal F(y;\cdot)$ is a holomorphic function on $\RR$. The measure $\tau$ on $\RR$ is called a {\it spectral measure of the problem} \eqref{prob} if $\mathcal F$ is an isometry from $\mathcal L_0\subset L^2(0,\ell)$ into the Hilbert space $L^2_\tau(\RR)$, that is, if the Parseval relation
\[
\int_0^\ell|y(x)|^2\dd x=\int_\mathbb R|\mathcal F(y;\la)|^2\dd \tau(\la),\quad y\in\mathcal L_0,
\]
holds.
 In this case the mapping $y\mapsto \mathcal F(y;\cdot)$ can be extended  by continuity to all of $L^2(0,\ell)$. The range of this extension is either the whole space $L^2_\tau(\RR)$ or a proper subspace of $L^2_\tau(\RR)$; correspondingly, $\tau$ is called an {\it orthogonal} or a {\it non-orthogonal spectral measure} of the problem \eqref{prob}.  
 %is an isometry from $L^2_0$ into $\calL^2_\tau$, the Hilbert space of the functions on $\RR$, which are square integrable with respect to the measure $\tau$.
 The set of all spectral measures of the problem \eqref{prob} is denoted by $\calS$, the set of all orthogonal spectral measures by $\calS^{\rm orth}$. It is well-known that $\calS$ contains exactly one element if the problem \eqref{prob} is singular and in limit point case at $\ell$;  this spectral measure is orthogonal. Otherwise, if \eqref{prob} is regular at $\ell$, or singular and in limit circle case at $\ell$, $\calS$ contains infinitely many orthogonal and infinitely many non-orthogonal spectral measures. For the case of a regular right endpoint a description of all the spectral measures was given, e.g.,  in \cite{KK1}, see also \eqref{bloch} below. 

If $0<a<\ell$ we consider the restriction of problem \eqref{prob} to $[0,a]$. This means that  $\ell$ is replaced by $a$, the potential of the restricted problem is the restriction  $q\big|_{[0,a]}$ of $q$, and $h  $ or $\al$ is the same as in \eqref{prob}. This is a regular problem, the set of all its spectral measures  is denoted by $\calS_a$. 
 It follows immediately from the definition of a spectral measure, that a spectral measure of the problem \eqref{prob} on the interval $[0,\ell)$ is also a spectral measure of the restricted problem on $[0,a]$.

%With each   $\tau\in\calS$, there is associated the  {\it Weyl--Titchmarsh function} $m_\tau$:
%\begin{equation}\nonumber%label{m}
%m_\tau(z)=\displaystyle\int_\RR \dfrac{\dd \tau(\la)}{\la-z},\quad z\in\CC^+\cup\CC^-.
%\end{equation}
%It belongs to the class $\boldN$ of Nevanlinna functions. 
Recall that a complex function $F$ is a {\it Nevanlinna function}, if it is holomorphic in $\CC^+\cup\CC^-$ and has the properties
\[
F(\ov z)=\ov{F(z)},\quad \Im F(z)\ge 0\  \text{ for } z\in\CC^+;
\]
the class of all Nevanlinna functions is denoted by $\boldN$, and we set  $\wt\boldN:=\boldN\cup\{\infty\}$. It is well known that $F\in\boldN$ if and only if $F$ admits a representation
\begin{equation}\label{nev}
F(z)=\al + \be z+\int_\RR\left(\dfrac 1{\la - z}-\dfrac\la{1+\la^2}\right)\dd \si(\la),\quad z\in\CC^+\cup\CC^-,
\end{equation}
where $\al\in\RR,\,\be\ge 0$, and $\si$ is a measure on $\RR$ with the property $\int_\RR\frac{\dd\si(\la)}{1+\la^2}<\infty$, called the {\it spectral measure of} $F$.

A description of the set $\calS_a$ of all spectral measures of a regular problem \eqref{prob} on $[0,a]$ can be obtained from the following result (see  \cite[Theorem 14.1]{KK1}).

%In the following we need a description of all Weyl-Titchmarsh functions $m_\tau$ or, equivalently, of all spectral measures $\tau$ of a regular problem \eqref{prob}, see ; we set.
\begin{itemize} 
\item[${2.1^{\rm o}}\!\!.$] 
  {\it If $\,0<a\le \ell$ and the problem \eqref{prob} is regular on $[0,a]$, then for $\ga\in\wt\boldN$ the function
\begin{equation}\label{res_for}
m_\ga(z):=\dfrac{\psi'(a;z)-\psi(a;z)\ga(z)}{\varphi'(a;z)-\varphi(a;z)\ga(z)},\quad z\in\CC\setminus\RR,
\end{equation}
is a Nevanlinna function: $m_\ga\in\boldN$. If $\tau_\ga$ denotes the spectral measure of $m_\ga$ then $\calS_a=\{\tau_\ga:\ga\in\wt\boldN\}$.  The spectral measure $\tau_\ga$ is orthogonal if and only if $\ga$ is a real constant or $\infty$.}
\end{itemize}

The function $m_\ga$ in \eqref{res_for} is the {\it Weyl--Titchmarsh function}  corresponding to the following (possibly $z$-depending) boundary condition at $x=a$:
\begin{equation}\label{bcc}
y'(a)-\ga(z)y(a)=0.
\end{equation}
In fact, the solution $y$ of the inhomogeneous problem
\begin{equation}\label{hc}
-y''+qy-zy=f \text{ on } [0,a],\quad
 y'(0)-   h\, y(0)=0,\ \ y'(a)-\ga(z)y(a)=0,
\end{equation}
can be written as
$y(x)=\int_0^aG(x,\xi;z)f(\xi)\dd\xi,\,0\le x\le a$, where
\[
G(x,\xi;z):=
\left\{\begin{array}{ll}\varphi(x;z)\big(m_\ga(z)\varphi(\xi;z)-\psi(\xi;z)\big),\quad& 0\le x\le\xi\le a,\\[1mm]
\varphi(\xi;z)\big(m_\ga(z)\varphi(x;z)-\psi(\xi;z)\big),\quad& 0\le \xi\le x\le a.
\end{array}\right.
\]
For the Weyl--Titchmarsh function $m_\ga\in\boldN$ and the corresponding spectral measure $\tau_\ga$ the relation \eqref{nev} specializes to
\[
 m_\ga(z)=-\cot\al+\int_\mathbb R\dfrac{\dd\tau_\ga(\la)}{\la - z}.
\] 
Combining this relation with \eqref{res_for} it follows that the set of all spectral measures of the problem\eqref{prob} is given through a fractional linear transformation with parameter $\ga\in\wt\boldN$:
\begin{equation}\label{bloch}
\int_\mathbb R\dfrac{\dd\tau_\ga(\la)}{\la - z}=\cot\al +\dfrac{\psi'(a;z)-\psi(a;z)\ga(z)}{\varphi'(a;z)-\varphi(a;z)\ga(z)},\quad z\in\CC\setminus\RR.
\end{equation}

If $\ga$ in \eqref{hc} is a real constant or $\infty$ with the problem \eqref{hc} there is defined a self-adjoint operator in the space $L^2(0,a)$, which we denote by $A_\ga$. Then, at least formally, with the delta-distribution $\delta_0$ the Parseval relation  implies
\begin{equation}\nonumber
\int_\mathbb R\dfrac{\dd\tau_\ga(\la)}{\la-z}=\dfrac 1{(\sin\al)^2}\big((A_\ga-z)^{-1}\delta_0,\delta_0\big)_{L^2(0,a)}.
\end{equation}

\medskip

\noindent{2.2.} In \cite{K52} with the spectral measure $\tau\!\in\! \calS$, besides the Weyl-Titchmarsh function $m_\tau$,  M.G.~Krein associates the {\it transfer function} $\Phi_\tau$ of the problem \eqref{prob} (see also \cite{M}):
\begin{equation}\label{trans}
\Phi_\tau(t):=\int_\RR\dfrac{1-\cos(\sqrt{\la}t)}{\la}\,\dd\tau(\la),\quad t\in [0,2\ell).
 \end{equation}
The integral in \eqref{trans} exists at least for $t\in [0,2\ell)$ (a proof will be given in Subsection 5.1), and the function $\Phi_\tau$ has an absolutely continuous second derivative there. 
Since for $a \in (0,\ell)$ a spectral measure of the problem \eqref{prob} is also a spectral measure of the restricted problem on $[0,a]$,  the restriction to $[0,2a]$ of a transfer function of   \eqref{prob}  is also a transfer function of the restricted problem on $[0,a]$. 

The expression on the right hand side of \eqref{trans} defines an extension of $\Phi_\tau$ to the interval $(-2\ell,2\ell)$ by symmetry: $\Phi_\tau(-t)=\Phi_\tau(t), \, t\in[0,2\ell)$, and possibly also to  an interval larger than $(-2\ell,2\ell)$. If, e.g., the support of $\tau$ is bounded from below, then $\Phi_\tau$ is defined by the integral in \eqref{trans} on $\RR$ and it is at most of exponential growth at $\infty$:
\[
|\Phi_\tau(t)|\le Ce^{\kappa t},\quad t\in \RR,
\]
for some $C,\,\kappa>0$.
 In this case,  for $z\in\CC^+$ with sufficiently large imaginary part we have
\begin{equation}\label{twt}
\int_0^\infty e^{\ii zt}\Phi_\tau(t)\dd t=\frac\ii z\int_\RR\frac{\dd\tau(\la)}{\la-z^2}=\dfrac\ii zm_\tau(z^2).
\end{equation}
If $\tau$ is an orthogonal spectral measure of a regular problem \eqref{prob} then the support of $\tau$ is bounded from below (see \cite[Satz 13.13]{W}) and \eqref{twt} holds.

The following properties of the transfer functions $\Phi_\tau$  of the problem  \eqref{prob} were formulated in \cite[Theorems 2 and 3]{K52}. 
\begin{itemize} {\it
\item[${2.2^{\rm o}}\!\!.$]    For $0\le t<2\ell$, the values $\Phi_\tau(t)$ do not depend on $\tau\in \calS_\ell$.\vspace*{3mm}
\item[${2.3^{\rm o}}\!\!.$]   The set of all spectral measures $\tau\in \calS_\ell$ coincides with the set of all measures $\tau$ for which a representation \eqref{trans} of the transfer function $\Phi_\tau$ holds on $[0,2\ell)$.
} 
\end{itemize}

Proofs of claim ${2.2^{\rm o}}\!\!$, and of claim ${2.3^{\rm o}}$ for orthogonal spectral measures will be given in Subsection 5.1 and 5.2 below.

The following fact is a crucial property of the transfer function of a Sturm--Liouville problem. It was obtained as an example for the method of directing functionals in \cite{Knuclei}, and is quoted in  \cite{K52} (see also  Subsection 5.3).

\begin{itemize} {\it
\item[${2.4^{\rm o}}\!\!.$]  A continuous functions $\Phi$ on $[0,2\ell)$ with $\Phi(0)=0$ admits  a representation \eqref{trans} with some measure $\tau$ on $\RR$:
\begin{equation}\label{fer}
\Phi(t)=\int_\RR\dfrac{1-\cos(\sqrt{\la}t)}{\la}\,\dd\tau(\la),\quad t\in [0,2\ell),
\end{equation}
if and only if the kernel 
\begin{equation}%\nonumber%label{kern}
K_\Phi(s,t):=\Phi(t+s)-\Phi(|t-s|),\quad 0\le s,t<\ell,
\end{equation}
is positive definite.  
}
\end{itemize}

 If the integral in \eqref{fer}  exists also  for $t\in[2\ell,2\wt \ell)$ with some $\wt \ell>\ell$, then the expression on the right hand side of \eqref{fer} defines a continuous continuation $\wt \Phi$ of $\Phi$ to the larger interval $[0,2\wt \ell)$ such that the kernel $K_{\wt\Phi}$ is positive definite on $[0,\wt \ell)$.

Statement ${2.3^{\rm o}}$ implies the following localization principle; here we write $(2.1_{\!j})$ for the problem $\eqref{prob}$ with parameters $\ell_j, q_j, h_j,\,j=1,2$.

\begin{theorem}\label{p21}
Suppose we are given two problems $(2.1_j)$ with corresponding transfer functions $\Phi_j$ on the intervals $[0,2\ell_j), \,j=1,2$. If, for some $a$ with $0<a\le\min\{\ell_1,\ell_2\}$,
\begin{equation}\label{fifi}
\Phi_1(t)=\Phi_2(t),\quad t\in[0,2a),
\end{equation}
then $q_1=q_2$ a.e. on $[0,a)$ and $h_1=h_2$.   
\end{theorem}

\begin{proof}
For short we write $(2.1_j^a)$ for the problem \eqref{prob} with parameters 
\[
 a, q_j\big|_{[0,a]}, h_j,\quad j=1,2.
\]
Then the restriction of $\Phi_1$ to $[0,2a)$ is the transfer function of problem $(2.1_1^a)$ and, since $\Phi_1(t)=\Phi_2(t)$ on $[0,2a)$, also the transfer function  of  problem $(2.1_2^a)$. Thus, by ${2.3^{\rm o}}$, the sets of spectral measures of the problems  $(2.1_1^a)$ and  $(2.1_2^a)$ coincide, and  the claim follows from the Borg--Mar\v cenko theorem (see, e.g. \cite[last paragraph]{K52}).
\end{proof}

For constant $\ga$, the value $\Phi_{\tau_\ga}(t)$ of the transfer function  has the following physical meaning  (see \cite{K52}). On the interval $[0,\ell)$ of the $x$-axis, consider a homogeneous string with mass density one, with an elastic foundation given by $q$, and the boundary conditions $y'(0)-hy(0)=0,\,h\ne\infty$, and  \eqref{bcc}. If the constant force $1$ starts to act at time $t=0$ perpendicularly to this string at the left endpoint $0$, then  $\Phi_{\tau_\ga}(t)$  is the position of the left endpoint  at time~$t$.

\begin{remark} In \cite{K52} the statements
${2.2^{\rm o}}$ and ${2.3^{\rm o}}$ are formulated for the more general problem
\begin{equation}\nonumber%label{probw}
-y''(x)+q(x)y(x)-z\rho(x)y(x)=0,\ x\in [0,\ell),\quad  y'(0)-   h y(0)=0,
\end{equation}
with a weight function $\rho\in L^1_{\rm loc}([0,\ell)),ß,\rho\ge 0$, which does not vanish on any sub-interval of $[0,\ell)$ of positive length. In this case, for $0<x\le\ell$  set
\[
a_x:=\int_0^x\sqrt{\rho(\xi)}\dd \xi.
\]
 Then the transfer functions from \eqref{trans} are defined at least  on $[0,2a_\ell)$, and if $\widehat\ell\in (0,\ell)$ for the corresponding restricted problem to $[0,\widehat\ell)$ the  transfer functions $\Phi_{\widehat\tau}$  coincide on the interval $[0,2a_{\widehat\ell}]$. 
\end{remark}

We conclude this subsection with plots of some transfer functions for two examples of Sturm-Liouville operators.\\[2mm]
Example 1. Consider the simplest problem
\[
-y''-zy=0,\quad x\in[0,\ell],\ y'(0)=0,\ y'(\ell)-\ga y(\ell)=0,
\]
with some $\ga\in\mathbb R\cup\{\infty\}$. The corresponding Weyl-Titchmarsh function is
\[
m_\ga(z)=\dfrac{-\cos(\sqrt{z}\ell)+\ga\dfrac{\sin(\sqrt{z}\ell)}{\sqrt{z}}}{\sqrt{z}\sin(\sqrt{z}\ell)+\ga\cos(\sqrt{z}\ell)};
\]
here we write $m_\ga$ instead of $m_{\tau_\ga}$ and, correspondingly, $\Phi_\ga$ instead of $\Phi_{\tau_\ga}$.

\begin{figure}[h]
\includegraphics[scale=0.75]{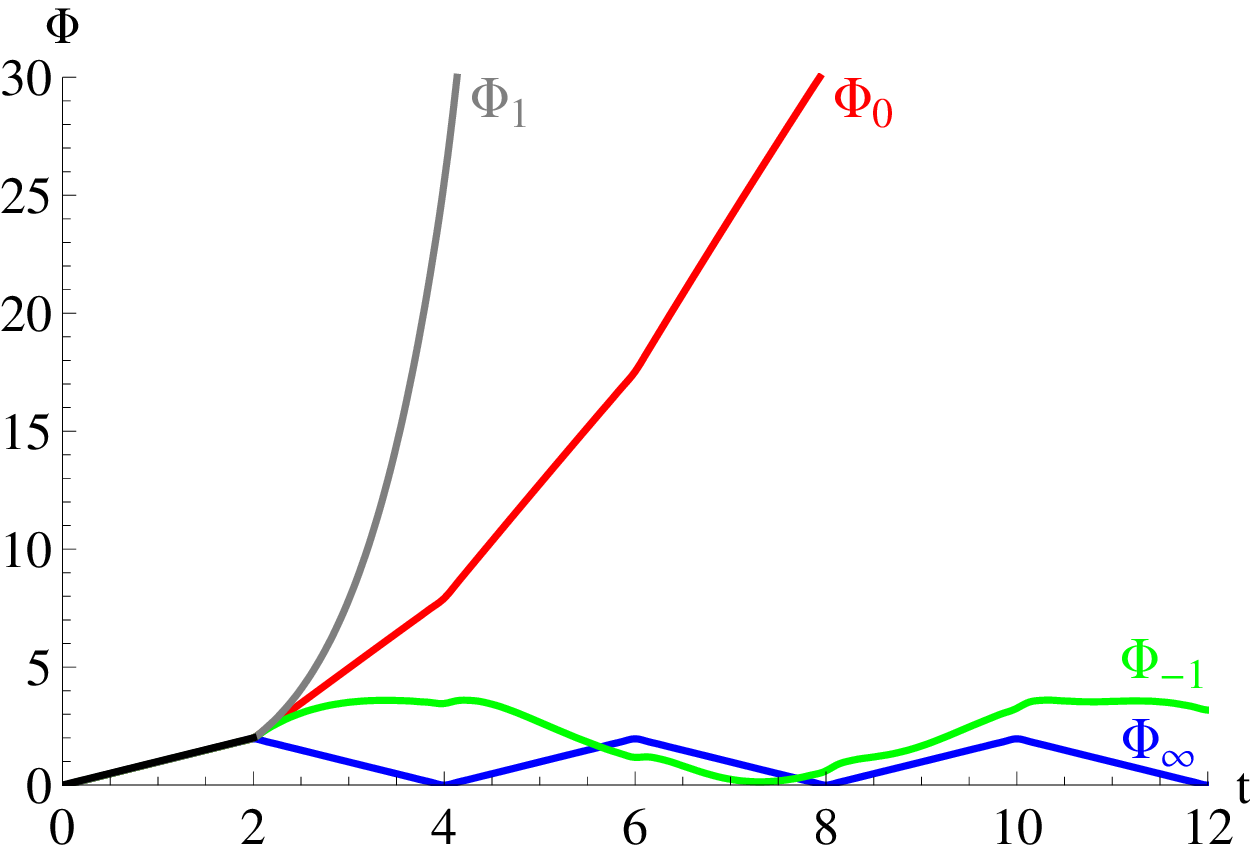}
\label{fig0}
\caption{}
\end{figure}

In Fig.\ 1  some transfer functions $\Phi_\ga$ are shown for $\ell=1$. They all coincide on $[0,2]$. The periodic function $\Phi_\infty$ corresponds to the Dirichlet, the function $\Phi_0$ to the Neumann boundary condition at $x=1$.\\[2mm]
Example 2. Consider the Bessel type problem
\begin{equation}\label{besss}
-y''(x)+\dfrac{2y(x)}{(x-1)^2}=zy(x),\quad y'(0)=0.
\end{equation}
On the interval $[0,1)$ it is singular and limit point at $x=1$, and we denote the corresponding transfer function by $\Phi_s$. We also consider \eqref{besss} on the interval $[0,\frac 12]$ and with a boundary condition $y'(\frac 12)-\ga y(\frac 12)=0$ at $x=\frac 12$; the corresponding transfer function is denoted by $\Phi_\ga$.

The fundamental system $\varphi,\,\psi$ of solutions of the differential equation in \eqref{besss} satisfying
$
\varphi(0,z)=1,\ \varphi'(0,z)=0,\
\psi(0,z)=0,\ \psi'(0,z)=1
$
 is 
\begin{align*}
\varphi(x,z)&=\frac 1{z}\left\{\left(\frac{ 1-zx}{x-1}\right)\dfrac{\sin (x\sqrt{z})}{\sqrt{z}}+\left(-\frac {x}{x-1}+z\right)\cos(x\sqrt{z})\right\},\\
\psi(x,z)&=\frac 1{z}\left\{\left(z-\frac {1}{x-1}\right)\dfrac{\sin (x\sqrt z)}{\sqrt{z}}+\left(\frac {x}{x-1}\right)\cos(x\sqrt{z})\right\};
\end{align*}
e.g. the Weyl-Titchmarsh function $m_s$ of the singular problem on $[0,1)$ becomes
\[
m_s(z)=\lim_{x\uparrow 1}\dfrac{\psi(x,z)}{\varphi(x,z)}=\dfrac{\sqrt z-\tan \sqrt z}{(1-z)\tan\sqrt z-\sqrt z}.
\]
Fig.\ \ref{fig2} shows the transfer functions $\Phi_s,\,\Phi_0,\Phi_\infty $.
\begin{figure}[h]
\includegraphics[scale=0.75]{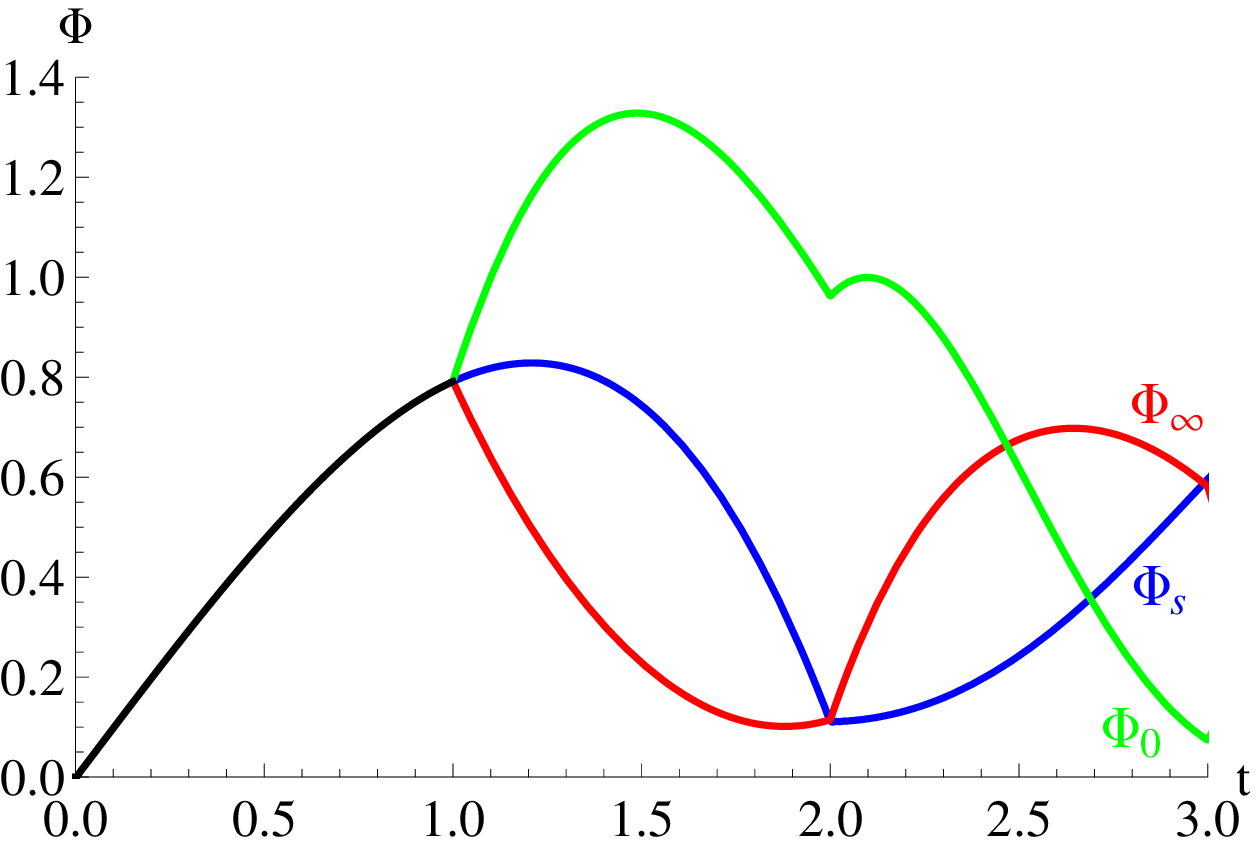}
\label{fig2}
\caption{}
\end{figure}

\medskip

\section{Localization by means of Weyl-Titchmarsh functions}

To obtain B. Simon's result it remains to formulate the condition \eqref{fifi} in terms of the corresponding Weyl-Titchmarsh functions $m_1$ and $m_2$. To do this we need a lemma and some well-known facts.

\medskip

\begin{lemma}\label{PHL}
Let $f\in L^1_{\rm loc}([0,\infty))$ be such that $|f(t)|={\rm O}\big(e^{y_0t}\big),\ t\to\infty$, for some $y_0\in \RR$. If, for some $\be>0$,
\begin{equation}\label{ray}
\int_0^\infty\!\!\! e^{\ii zt}f(t)\dd t={\rm O}\big(e^{-\be\Im z}\big),\ z\!\to\!\infty\ \text{ along some ray  }\ 0<{\rm arg}z<\pi/2,
\end{equation}
then $f(t)=0$ a.e. on $[0,\be]$.
\end{lemma}

\begin{proof}\hspace*{-1.5mm}\footnote{I thank Professor Vadim Tkachenko for communicating this proof to me.}
Define $F(z):=e^{-\ii \be z}\int_0^\be e^{\ii zt}f(t)\dd t$. It is an entire function of exponential type. With $z=x+\ii y$, the relation
\[
F(z)=e^{-\ii  \be x}\int_0^\be e^{\ii xt}e^{(\be-t)y}f(t)\dd t
\]
shows that $F$ is bounded in the lower half plane $\CC^-,$ the relation
\begin{align*}
e^{-\ii \be z}\int_0^\be e^{\ii zt}f(t)\dd t&=e^{-\ii \be z}\left(\int_0^\infty e^{\ii zt}f(t)\dd t-\int_\be^\infty e^{\ii zt}f(t)\dd t\right)\\
&=e^{-\ii \be z}\int_0^\infty e^{\ii zt}f(t)\dd t- e^{-\ii\be x}\int_\be^\infty e^{\ii xt}e^{-y(t-\be)}f(t)\dd t
\end{align*}
implies that $F$ is bounded on the ray in \eqref{ray}. According to the Phragm$\acute{\rm e}$n--Lindel\"of principle, this yields $F(z)=const$. The Riemann-Lebesgue lemma, applied to $F(z)=\int_0^\be\ee^{-\ii zt}f(\be-t)\,\dd t$, gives $F(z)=0$ and, finally, $f(t)=0$ a.e. on $[0,\be]$ 
\end{proof}

References for the following statement can be found e.g. in \cite{B}.

\begin{itemize} 
\item[${3.1^{\rm o}}\!\!.$] {\it If $0<a\le \ell$ and the problem \eqref{prob} is regular on $[0,a]$, then the  asymptotic relation 
\begin{equation}\nonumber
\varphi(a;z)=\frac 12\left(\sin\al+\frac{\cos\al}{\sqrt{-z}}\right)e^{\,a\sqrt{-z}}\big(1+{\rm o}(1)\big)
\end{equation} 
holds for $z\to\infty$ along any non-real ray; here the square root is the principal root, that is the root with positive real part.}
\end{itemize}

\smallskip

The next claim follows from the integral representation \eqref{nev}.

\begin{itemize}{\it
\item[${3.2^{\rm o}}\!\!.$] 
 For a Nevanlinna function $g$ we have
\[
g(z)={\rm O}\,(|z|),\quad z\to\infty\ \text{ on any non-real ray}.
\]}
\end{itemize}

\medskip

Now we return to the regular problem \eqref{prob} on $[0,a]$. It follows from ${2.1^{\rm o}}$, that for $\ga,\wh\ga\!\in\!\boldN$ and the corresponding Weyl--Titchmarsh functions $m$ and $\wh m$ we have
\begin{equation}\label{diff}
\begin{array}{rcl}
m(z)-\wh m(z)\!\!\!\!&\!\!=\!\!&\!\!\!\!\dfrac{\psi'(a;z)-\psi(a;z)\ga(z)}{\varphi'(a;z)-\varphi(a;z)\ga(z)}-\dfrac{\psi'(a;z)-\psi(a;z){\wh\ga}(z)}{\varphi'(a;z)-\varphi(a;z){\wh\ga}(z)}\\[4mm]
&=&\dfrac{\ga(z)-\wh\ga(z)}{\varphi(a;z)^2\left(\frac{\varphi'(a;z)}{\varphi(a;z)}-\ga(z)\right)\left(\frac{\varphi'(a;z)}{\varphi(a;z)}-{\wh\ga}(z)\right)}.
\end{array}
\end{equation}
Since $\ga,\,\wh\ga$, and  $-\frac{\varphi'(a;\cdot)}{\varphi(a;\cdot)}$  are Nevanlinna functions (comp. \cite[$\S$ 2.4]{KK1}), so are $\left(\!\frac{\varphi'(a;z)}{\varphi(a;z)}-\ga(z)\!\right)^{-1}$ and $\left(\!\frac{\varphi'(a;z)}{\varphi(a;z)}-\wh\ga(z)\!\right)^{-1}$. The relation \eqref{diff} and the statements ${3.1^{\rm o}}$ and ${3.2^{\rm o}}$ imply
 \begin{equation}\label{ee}
m(z)-\wh m(z)={\rm O}\!\left(e^{-2a(1-\varepsilon)\Re\sqrt{-z}}\right),\quad z\to\infty \text{ on any non-real ray},
\end{equation} 
for all $\varepsilon>0$.

Finally, we can prove the following theorem of B. Simon (\cite{Si1}). For short we write $m_j:=m_{\tau_j},\ j=1,2.$ 

\begin{theorem}\label{main}
Consider  two problems $(2.1_j)$ as in Theorem \ref{p21}. Let $a$ be such that $0<a<\min\{\ell_1,\ell_2\}$, and suppose that for a spectral measure $\tau_1$ of the problem $(2.1_1)$ and a spectral measure $\tau_2$ of the problem $(2.1_2)$ we have
\begin{equation}\label{w2}
m_1(z)\!-m_2(z)={\rm O}\!\left(e^{-2a(\!1-\varepsilon)\Re\!\sqrt{-z}}\right),\quad z\!\to\!\infty \text{  on some non-real ray,}
\end{equation}
for all  $\varepsilon>0$.Then $€q_1=q_2$ a.e. on $[0,a]$ and $h_1=h_2$.
\end{theorem}

\begin{proof}
The claim follows from Theorem \ref{p21} if we show that \eqref{w2} implies that 
\begin{equation}\label{nn}
\Phi_1(t)=\Phi_2(t),\quad t\in[0,2a].
\end{equation}
To this end, if the boundary condition \eqref{bcc} which corresponds to  the spectral measure $\tau_j$ depends on $z$, we replace this boundary condition at $a$ by a boundary condition where $\ga$ is a real constant, e.g. by $y'(a)=0$. To this problem there corresponds a new spectral measure $\wh\tau_j$, such that between the corresponding Weyl-Titchmarsh functions 
$m_j$ and $\wh{m}_j$ the relation \eqref{ee} holds, $j=1,2$.
Together with \eqref{w2} this implies that
\begin {align*}
\wh m_1(z)-\wh m_2(z)&=\wh m_1(z)-m_1(z) +m_1(z)-m_2(z)+ m_2(z)-\wh m_2(z)\\
&={\rm O}\left(e^{-2a(1-\varepsilon)\Re\sqrt{-z}}\right),\quad z\to\infty \text{  on some non-real ray,}
\end{align*}
for all $\varepsilon >0$. Since the support of $\wh\tau_j$ is bounded from below (see \cite[Satz 13.13]{W}), the corresponding transfer functions $\wh\Phi_1$ and $\wh\Phi_2$ are defined on the whole real axis and are of exponential growth at $\infty$. Therefore \eqref{twt} holds and we find
\begin{align*}
\int_0^\infty e^{\ii zt}\wh\Phi_1(t)\dd t&-\int_0^\infty e^{\ii zt}\wh\Phi_2(t)\dd t=\dfrac\ii z \wh m_1(z^2)-\dfrac\ii z \wh m_2(z^2)\\
&={\rm O}\left(e^{-2a(1-\varepsilon)\Im z}\right),\quad z\to\infty\text{  on some non-real ray,}
\end{align*}
for all $\varepsilon >0$. Lemma \ref{PHL} yields $\wh\Phi_1(t)\!=\wh\Phi_2(t),\ t\!\in\![0,2a]$, and since $\wh\Phi_j(t)\!=\!\Phi_j(t),\ t\!\in\! [0,2a]$, by  ${2.2^{\rm o}}$, the relation \eqref{nn} follows.
\end{proof}

\section{Transfer functions and local spectral uniqueness for canonical systems and strings} 

\noindent{4.1.}\  For $2$-dimensional canonical systems the role of the transfer functions is played by screw functions, see \cite{KL5}. To explain this, we consider the following canonical system with a symmetric boundary condition at $x=0$:
\begin{equation}\label{ks}
-J{\bf y}'(x)=zH(x){\bf y}(x),\ x\in[0,\ell),\ z\in\mathbb C,\quad {\bf y}(0)\in {\rm span} \{(0\ 1)^{\rm t}\}; 
\end{equation}
  here $0<\ell\le\infty$, the {\it Hamiltonian} $H=\left(h_{ij}\right)_{i,j=1}^2$ is supposed to be a real symmetric non-negative measurable $2\times 2$-matrix function on $[0,\ell)$ which is {\it trace normed}, that is  ${\rm tr}\,H(x)\!=\!1,\,x\in[0,\ell)$, a.e., and satisfies the condition $\int_0^xh_{22}(\xi)\dd\xi>0$ if $x>0$. 

The spectral measures for problem \eqref{ks} are defined as follows (see, e.g. \cite{KL5}).  Consider the solution $W(x;z)$ of the matrix differential equation
\begin{equation}\label{21}
\dfrac{\dd W(x;z)}{\dd x}J=zW(x;z)H(x),\quad W(0;z)=I_2,\quad 0\le x<\ell,\ z\in \mathbb C,
\end{equation}
where for $\ell<\infty$ also $x=\ell$ is allowed. Then, if $\ell<\infty$,  for arbitrary $\ga\in\wt\boldN$ the corresponding {\it Weyl-Titchmarsh function}
\begin{equation}\nonumber% label{e30}
W_{\langle\ga\rangle}^{(\ell)}(z):=\dfrac{w_{11}(\ell;z)\ga(z)+w_{12}(\ell;z)}{w_{21}(\ell;z)\ga(z)+w_{22}(\ell;z)}
\end{equation}
belongs to $\boldN$ and the spectral measures of all these functions $W_{\langle\ga\rangle}^{(\ell)},\,\ga\!\in\!\wt\boldN$, are by definition the {\it spectral measures $\tau_\ga$  of the problem} \eqref{ks}. If $\ell=\infty$, this problem has a unique spectral measure namely the spectral measure of the Nevanlinna function
\[
z\mapsto\lim_{x\to\infty}W_{\langle\ga\rangle}^{(x)}(z),\quad z\in\mathbb C^+\cup\mathbb C^-,
\]
which is independent of $\ga\in\wt\boldN$. An equivalent definition of the spectral measures of 
 \eqref{ks} by means of the Fourier transformation can be given, see \cite{K} and also  \cite{KL5}. If $0<l\le \ell$  the set of all spectral measures of the problem \eqref{ks} on $[0,l]$ is denoted by $\calS_l^c$.

For any measure $\tau$ on $\RR$ with 
\begin{equation}\label{inbed}
\int_\RR\frac{\dd\tau(\la)}{1+\la^2}<\infty
\end{equation} 
and numbers $\al,\be\in\RR$ a {\it screw function} $g(t),\,t\in\RR$, is defined by the formula
\begin{equation}\label{al}
 g(t):=\al+\ii\be t+ \int_{-\infty}^\infty \left(e^{\ii\la t}-1-\dfrac{\ii\la t}{1+\la^2}\right)\dfrac{\dd\tau(\la)}{\la^2},\quad t\in\RR.
\end{equation}
It has the characteristic property, that it is continuous and the kernel
\begin{equation}\label{kerG}
G_g(s,t):=g(s-t)-g(s)-\overline{g(t)}+g(0),\quad s,t\in\RR,
\end{equation}
is positive definite. The measure $\tau$ in the representation \eqref{al} is called the {\it spectral measure of} $g$. Evidently, in the representation \eqref{al} we have $\al=g(0)$. In the following we consider only screw functions $g$ with $g(0)=0$; this class is denoted by $\calG$.

If $\tau$ is a spectral measure of the problem \eqref{ks} for any $\be\in\mathbb R$ a corresponding {\it transfer function} $g_\tau$ of  \eqref{ks} is defined as the screw function
\begin{equation}\label{55}
 g_\tau(t):= \ii\be t+\int_{-\infty}^\infty \left(e^{\ii\la t}-1-\dfrac{\ii\la t}{1+\la^2}\right)\dfrac{\dd\tau(\la)}{\la^2},\quad t\in\RR.
\end{equation}
In contrast to the transfer function for a Sturm--Liouville problem, the function $g_\tau\in\calG$ in \eqref{55} is always defined on the whole real axis. According to a basic result of L.\,De~Branges  \cite{DB}, see also \cite{Wi}, each measure $\tau$ on $\RR$ with the property \eqref{inbed} is the spectral measure of a unique canonical system \eqref{ks} on $[0,\infty)$, and hence also every function of the form \eqref{55} is the transfer function of a unique canonical system.

For a canonical system  \eqref{ks} we set
\begin{equation}\label{al1}
a(l):=\int_0^{l}\sqrt{\det H(x)}\dd x,\quad l\in[0,\ell).
\end{equation}
Then the following statements hold.

\medskip  

\begin{itemize} {\it
\item[${4.1^{\rm o}}\!\!.$]   Suppose that $0<l<\ell$ and $a(l)>0$. If $\tau_1,\tau_2\in \calS_{l}^c$ then for any two screw functions $g_{\tau_1},\,g_{\tau_2}$ for the difference of the restrictions $g_{\tau_1}\big|_{[0,2a(l)]}$ and $g_{\tau_2}\big|_{[0,2a(l)]}$ it holds
\begin{equation}\nonumber%label{sp}
 g_{\tau_1}(t)-g_{\tau_2}(t)=\ii\beta t,\quad t\in [0,2a(l)],
\end{equation}
with some $\be\in\mathbb R$.
%\end{itemize}

%\begin{itemize}{\it
\item[${4.2^{\,\rm o}}\!\!.$] Suppose that $0<l<\ell$ and $\int_{l-\varepsilon}^{l}\sqrt{\det H(x)}\,\dd x>0$ for all  $\varepsilon>0$. If $\tau\in\calS_l^c$ and $g_\tau$ is a corresponding screw function, then the set of all spectral measures $\calS_{l}^c$  of the canonical system coincides with the set of spectral measures of all the continuations of  $g_\tau\big|_{[0,2a(l)]}$ in the class $\calG$.  }
\end{itemize}

The statement ${4.2^{\,\rm o}}$ follows from \cite[Theorem 5.6]{KL5}. To prove ${4.1^{\rm o}}$, consider $\tau_1\in\calS_l^c$ with a corresponding screw function $g_{\tau_1}$. If $\tau_2\in\calS_l^c$, according to  ${4.2^{\,\rm o}}$ there exists a continuation $\wt g_{\tau_1}\!\in\!\calG$ of $g_1\big|_{[0,2a(l)]}$ with spectral measure~$\tau_2$:
\begin{equation}\nonumber
{\wt g}_{\tau_1}(t)= \ii\wt\be_2 t+\int_{-\infty}^\infty \left(e^{\ii\la t}-1-\dfrac{\ii\la t}{1+\la^2}\right)\dfrac{\dd\tau_2(\la)}{\la^2},\quad t\in\RR,
\end{equation}
$\wt\be_2$ real. On the other hand, with some real $\be_2$,
\[
 g_{\tau_2}(t)= \ii\be_2 t+\int_{-\infty}^\infty \left(e^{\ii\la t}-1-\dfrac{\ii\la t}{1+\la^2}\right)\dfrac{\dd\tau_2(\la)}{\la^2},\quad t\in\RR,
\]
and it follows that
\[
g_{\tau_1}(t)-g_{\tau_2}(t)=\wt{g}_{\tau_1}(t)-g_{\tau_2}(t)=\ii (\wt{\be}_2-\be_2).
\]

\medskip
Here are some transfer functions for two examples of canonical systems.\\[2mm]
Example 3. Consider the  Hamiltonian
\[
H(x)=\begin{pmatrix}\frac 12&0\\0&\frac 12\end{pmatrix},\quad 0\le x\le \ell.
\]
Then $\det H(x)=\frac 14,\ \int_0^\ell\sqrt{\det H(x)}\dd x=\frac\ell 2$,
\[
W(\ell;z)=\begin{pmatrix}\cos\left(\frac\ell 2z\right)&\sin\left(\frac\ell 2z\right)\\[1mm]
-\sin\left(\frac\ell 2z\right)&\cos\left(\frac\ell 2z\right)\end{pmatrix},
\]
and for $\ga\in\RR \cup\{\infty\}$ we obtain the Weyl-Titchmarsh function 
\[
m_\ga(z)=\dfrac{\cos\left(\frac\ell 2z\right)\ga+\sin\left(\frac\ell 2z\right)}{-\sin\left(\frac\ell 2z\right)\ga+\cos\left(\frac\ell 2z\right)}.
\]

We choose $\ell=2$, and suppose first $\ga\in\RR$. Then the eigenvalues are $\la_k=\la_0+k\pi$ with $\la_0\in (-\pi/2,\pi/2]$ and $k\in\ZZ$ and with corresponding spectral measure  $\tau_k=1$. With some real $\be$ the corresponding transfer function becomes
\begin{align*}
g_\ga(t)=&\ii\be t+\int_\RR\left(\ee^{\ii\la t} -1-\frac{\ii\la t}{1+\la^2}\right)\frac {\dd\tau(\la)}{\la^2}\\
=&\ii\be t+\sum_{k\in\ZZ}\left(\ee^{\ii t(\la_0+k\pi)}  - 1-\frac{\ii\la_kt}{1+\la_k^2}\right)\frac 1{\la_k^2}\\
=&\ii\be t+\ee^{\ii t\la_0}\sum_{k\in\ZZ}\frac{\ee^{\ii t k\pi}}{(\la_0+k\pi)^2}\\
&-\sum_{k\in\ZZ}\frac {1}{\la_k^2}-\ii t\sum_{k\in\ZZ}\frac {1}{(1+\la_k^2)\la_k}.
\end{align*}
Setting $x=\frac{t\pi}2$ the Fourier series in the second last line becomes
\[ 
\sum_{k\in\ZZ}\frac{\ee^{\ii  k\pi t}}{(\la_0+k\pi)^2}=\sum_{k\in\ZZ}\frac{\ee^{2\ii kx}}{(\la_0+k\pi)^2}=\sum_{j\in\ZZ}c_j\ee ^{\ii jx},\quad c_j=\left\{\!\!\!
\begin{array}{cl}\frac 1{\left(\la_0+\frac j2\pi\right)^2}\ \  &j \text{ even,}\\
0&j \text{ odd.}
\end{array}\right.
\]
With Mathematica it can be shown that it equals
\[
\left[\ee^{-2\ii \la_0x/\pi}\left(\frac 2\pi\left(-|x|+\frac\ii{\tan \la_0}x\right)+\frac 1{(\sin\la_0)^2}\right)\right]_{(-\pi ,\pi ]},
\]
where for a function $u$ defined on some bounded interval $(\al,\be]$, $[u]_{(\al,\be]}$  denotes the periodic continuation of $u$ to the real axis. Using the relations 
\[
\sum_{k\in\ZZ}\frac {1}{\la_k^2}=\sum_{k\in\ZZ}\frac {1}{(\la_0+k\pi)^2}=\frac 1{(\sin\la_0)^2},
\]
and
\[
\sum_{k\in\ZZ}\frac {1}{(1+\la_k^2)\la_k}=\frac 1{\tan\la_0}-\frac{\sin(2\la_0)}{\cosh 2-\cos(2\la_0)},
\]
we can write
\begin{align*}
g_\ga(t)
=&\ee^{\ii t\la_0}\left[\ee^{-\ii t\la_0}\left(-|t|+\frac{\ii t}{\tan \la_0}+\frac 1{(\sin\la_0)^2}\right)\right]_{(-2,2]}\\
&-\frac 1{(\sin\la_0)^2}+\ii t\left(\be -\frac 1{\tan\la_0}+\frac{\sin(2\la_0)}{\cosh 2-\cos(2\la_0)}\right).
\end{align*}
%Now we observe that 
%\[
%\sum_{k\in\ZZ}\frac {1}{\la_k^2}=\sum_{k\in\ZZ}\frac {1}{(\la_0+k\pi)^2}=\frac 1{(\sin\la_0)^2},
%\]
%and
%\[
%\sum_{k\in\ZZ}\frac {1}{(1+\la_k^2)\la_k}=\ga-\frac{\sin(2\la_0)}{\cosh 2-\cos(2\la_0)}.
%\]
Hence if $\be$ is chosen as 
\[
\be=-\frac{\sin(2\la_0)}{\cosh 2-\cos(2\la_0)},
\]
then
\[
g_\ga(t)=-|t|,\quad -2\le t\le 2,
\]
independent of  $\ga$.

If $\ga=\infty$ then 
$m_\infty(z)=-\cot\left(\frac\ell 2z\right),\ \la_k=k\pi,\,k\in\ZZ$, %: $\la_k=\frac{2\pi k}{\ell},\ \tau_k=\frac 2\ell,\quad k=0,\pm1, \pm2,\dots$
 and 
\begin{align*}
g_\infty(t)&=\int_\mathbb R\left(e^{\ii\la t}-1-\frac{\ii\la t}{1+\la^2}\right)\dfrac{\dd\tau(\la)}{\la^2}\\
&=\sum_{k=1}^\infty 2(\cos(\la_kt)-1)\frac{\tau_k}{\la_k^2} -\frac{t^2}{2}\tau_0
=\left[\frac{t^2}2-|t|\right]_{[-2,2]}-\frac{t^2}2.
\end{align*}
%where $\left[\dfrac{t^2}\ell-t\right]_{[-\ell,\ell]}$ denotes the periodic continuation to $\mathbb R$ of the restriction to $[-\ell,\ell]$ of the function in brackets. Here we have used the relation
%\[
%\frac{\pi^2}{12}+\sum\limits_1^\infty \frac{\cos(k\xi)}{k^2}=\frac{(\xi-\pi)^2}{4},\quad 0\le\xi\le 2\pi.
%\]
All these screw functions are piecewise linear. The functions $g_0$ and $g_\infty$ and the real parts of $g_1$ and $g_2$ are plotted in Fig. 3.

\begin{figure}[h]
\includegraphics[scale=0.75]{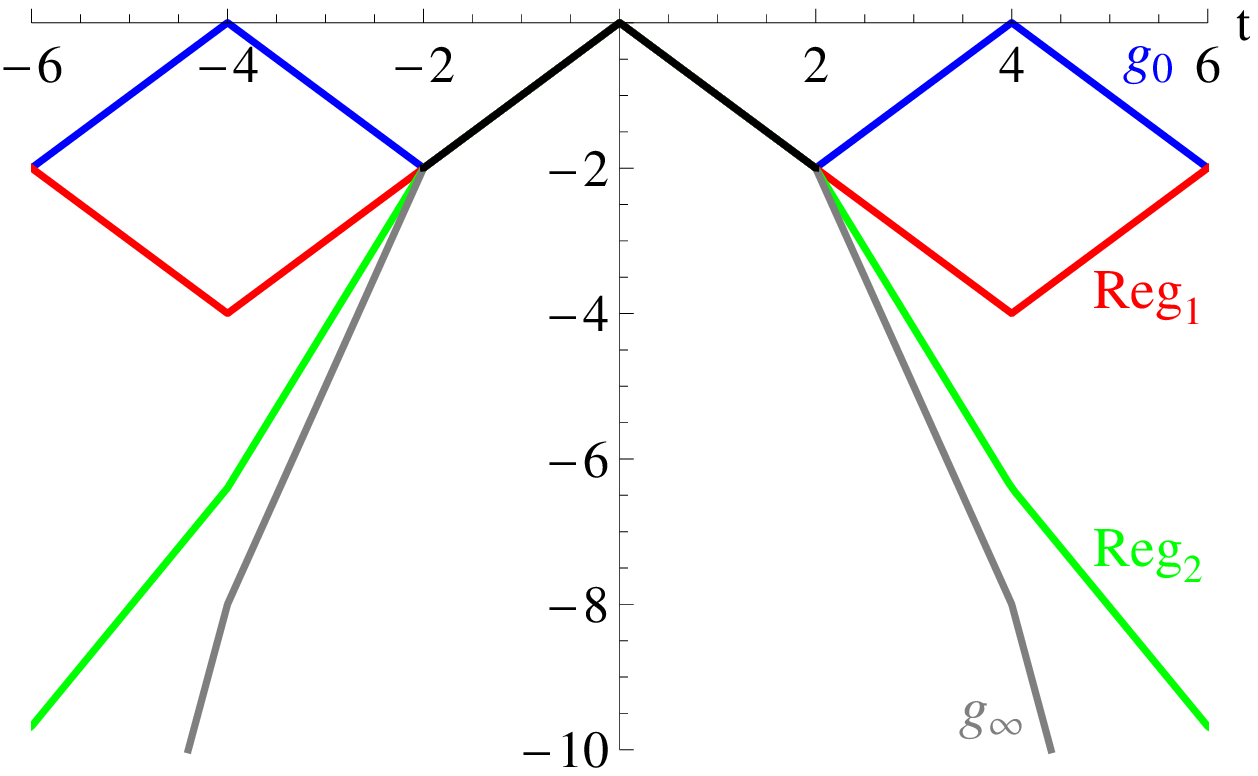}
\label{fig3}
\caption{}
\end{figure}

\medskip

\noindent Example 4. Consider the Hamiltonian
\[
H(x)=\begin{pmatrix}(x-1)^2&0\\0&(x-1)^{-2}\end{pmatrix},\quad 0\le x<1.
\]
It is not trace-normed. With the new  variable $\xi(x):={\rm tr}\,H(x),\ 0\le x<1$, functions $\wt{\bf y}(\xi(x))={\bf y}(x),\ 0\le x<1$, and the new Hamiltonian ${\wt H}(\xi(x)):=H(x),\ 0\le x< 1$,
it becomes the trace-normed system $-J{\bf {\wt y}}'=z{\wt H}{\bf {\wt y}}$ on $[0,\infty)$. Then 
\[
\int_0^{\xi(x)}\sqrt{\det {\wt H}(\xi)}\dd \xi=\int_0^x\sqrt{\det {H}(x)}\dd x=x,\quad 0\le x<1.
\]
The solution $W$ of \eqref{21} is 
\[
W(x;z)=\begin{pmatrix}\dfrac{\sin (xz)\! -\! z\cos (xz)}{z(x\!-\!1)}&\left(\dfrac 1{z^2}\!-\!(x\!-\!1)\right)\sin (xz)\!-\!\dfrac{x\cos (xz)}z\\[2mm]
\dfrac{\sin (xz) }{(x\!-\!1)}&\dfrac{\sin (xz)}z \!-\! (x\!-\!1)\cos (xz)
\end{pmatrix},
\]
and for the Weyl-Titchmarsh function for the singular problem  on $[0,1)$ we find
\[
m(z)=\dfrac{\sin z-z\cos z}{z\sin z}=\dfrac 1z-\cot z.
\]
The eigenvalues are $\la_k=k\pi,\ k=\pm 1,\pm2,\dots$, with spectral weights $\tau_k=1$, and  as a transfer function  we obtain
\begin{align*}
g(t)\!=\!\int_\mathbb R \left(\ee^{\ii\la t}\!-\!1\right)\dfrac{\dd\tau(\la)}{\la^2}
\!=\!\sum_{k=1}^\infty2\left(\cos k\pi t\!-\!1\right)\dfrac1{k^2\pi^2}
\!=\!\left[\dfrac{t^2}2 -t\right]_{(0,2]},\quad t\in\RR.
\end{align*}

\medskip

%\end{document}

A  localization principle for the problem \eqref{ks} by means of transfer functions can be formulated as follows; for simplicity we consider canonical systems  on the whole half axis $[0,\infty)$ (see \cite[${ 6.1^{\rm o}}$]{KL5}).

\begin{itemize} {\it
\item[${4.3^{\rm o}}\!\!.$]
 Let the Hamiltonians $H_1,\,H_2$ on  $[0,\infty)$ satisfy the same assumptions as $H$ at the beginning of this section and denote by  $g_1, g_2$  corresponding transfer functions.  Suppose that  $a_1:=\int_0^{\infty}\sqrt{\det H_1(x)}\dd x>0$. If \ $0<a<a_1$ and 
\[
l(a):=\inf\left\{l:\int_0^{l}\sqrt{\det H_1(x)}\,\dd x=a\right\},
\]
 then
\[
 H_1\big|_{[0,l(a)]}= H_2\big|_{[0,l(a)]}, \text{ a.e.},
\]
if and only if 
\begin{equation}\nonumber%label{gh}
g_1(t)-g_2(t)= \ii\be t,\quad t\in [0,2a],
\end{equation}
for some real number $\be$.} 
 \end{itemize}

Observe that here only on  intervals $[0,b]$, such that 
\[
\int_{b-\epsilon}^b\sqrt{\det \,H(x)}\,\dd x>0
\]
for all $\varepsilon >0$, the Hamiltonian $H$ is determined by its transfer functions.
\smallskip

A localization principle for canonical systems by means of their Weyl-Titchmarsh functions was proved in  \cite{LW}.

\medskip

\noindent{4.2.} If, for some $\mu>0$,  the Hamiltonian $H$ in \eqref{ks} is of the form  
\begin{equation}\nonumber%label{mass}
H(x)=\begin{pmatrix}0&0\\0&1\end{pmatrix},\quad 0\le x\le\mu,
\end{equation}
 then the spectral measures of the problems \eqref{ks} are finite (see, e.g., \cite{KL5}). In this case, instead of the screw functions $g_\tau$  the functions 
\begin{equation}\label{49}
f_\tau(t):=\int_{-\infty}^\infty e^{\ii\la t}\dd\tau(\la),\quad t\in\RR,
\end{equation}
 can be introduced. The statements ${4.1^{\rm o}}-{4.3^{\,\rm o}}$ remain true with $g_\tau$ replaced by $f_\tau$ and $\be=0$, see \cite{KL5}. By Bochner's theorem, the characteristic property of a continuous function $f$ to have the  representation \eqref{49} is that the kernel
\begin{equation}\label{efef}
F_f(s,t):= f(s-t),\quad s,t\in\RR,
\end{equation}
is positive definite.\\[2mm]
\noindent Example 5 (comp. \cite{LLS}). A slight alteration of Example~4~is~the~Hamiltonian
\[
H(x)=\left\{\begin{array}{cl}\begin{pmatrix}0&0\\0&1\end{pmatrix},\quad & 0\le x<1,\\[4mm]
\begin{pmatrix}(x-2)^2&0\\0&(x-2)^{-2}\end{pmatrix},\quad & 1\le x\le 2.\end{array}\right.
\]
Then
\[
W(x;z)\!=\!\left\{\!\!\begin{array}{ll}
\begin{pmatrix}1&0\\-zx&1\end{pmatrix},&0\le x<1,\\[2mm]
\begin{pmatrix}1&0\\-z&1\end{pmatrix}\begin{pmatrix}\dfrac{\sin ((x-1)z)\! -\! z\cos ((x-1)z)}{z(x\!-\!2)}&\cdots\\[2mm]
\dfrac{\sin ((x-1)z) }{(x\!-\!2)}&\cdots
\end{pmatrix},&1\le x<2
\end{array}\right.
\]
and the Weyl-Titchmarsh function for the singular problem on $[0,2)$ becomes
\[
m(z)=\dfrac 1{z^2}\tan z-\dfrac 1z,
\]
and the eigenvalues are $\la_k=(2k-1)\frac \pi 2$ with corresponding spectral masses $\tau_k=\dfrac 1{\la_k^2},\ k=\pm 1,\pm 2, \dots$. It follows that
\begin{align*}\nonumber
f(t)=\int_\mathbb Re^{\ii\la t}&\dd\tau(\la)=\sum_{k=1}^\infty\frac{2\cos(\la_kt)}{\la_k^2}\\
&=8\sum_{k=1}^\infty\dfrac{\cos(2k-1)\frac\pi 2t}{(2k-1)^2\pi^2}=[1-|t|]_{[-2,2]},\quad t\in\RR.
\end{align*}

\medskip

\noindent 4.3. Consider again the system in  \eqref{ks}, and suppose that  there exists an $l_0,\ 0<l_0\le\ell$, such that the Hamiltonian $H$ has the property 
\begin{equation}\label{pot}
\det H(x)>0,\quad x\in [0,l_0), \text{  a.e.}
\end{equation}
Then the function $l\mapsto a(l)$ in \eqref{al1} is continuous and strictly increasing on $[0,l_0]$. It is the inverse of the mapping $a\to l(a)$ on $[0,a(l_0)]$. Now ${4.3^{\rm o}}$ implies:

\begin{itemize}
\item[${4.4^{\rm o}}$]{\it
If the Hamiltonian $H$ in \eqref{ks} satisfies  \eqref{pot}, then for each $l\in [0,l_0)$ the values of the Hamiltonian $H$ on $[0,l)$ are $($a.e.$)$ uniquely determined by the values of a corresponding transfer function $g$ on $[0,2a(l)]$.}

\end{itemize}

 The assumption \eqref{pot} is satisfied if the canonical system can be written with a potential $V$:
\begin{equation}\label{pot1}
-J{\bf y}'(x)=z{\bf y}(x)+V(x){\bf y}(x),\ x\in[0,\ell'),\quad {\bf y}(0)\in {\rm span} \{(0\ 1)^{\rm t}\}, 
\end{equation}
where $V$ is a real symmetric  $2\times 2$--matrix function on $[0,\ell')$ which is locally summable there. With the $2\times 2$ matrix function $U$ on $[0,\ell')$, which is the solution of the initial problem
\[
\frac{\dd U}{\dd x}J=U(x)V(x),\quad x\in  [0,\ell'),\quad U(0)=I_2,
\]
we introduce a function ${\bf w}$ by ${\bf y}(x)=U(x){\bf w}(x),\ x\in[0,\ell')$. Since
\[
U(x)JU(x)^*=J,\quad x\in[0,\ell'),
\]
 it follows easily that $\bf w$ satisfies the canonical equation \eqref{ks} with $H(x)=U(x)U(x)^*$. This Hamiltonian is real, continuous, and  ${\rm det}\,H(x)=1,\ x\in[0,\ell')$,  but in general $H$ is not trace normed. However, by a  change of the independent  variable it can be transformed into a trace normed system of the form \eqref{ks} which satisfies the assumption \eqref{pot}. Therefore the conclusions of statement ${4.4^{\rm o}}$ apply to the problem \eqref{pot1}.

\medskip

\begin{remark} Suppose that the transfer function $g$ in \eqref{55} has a continuous {\it  accelerant} $h$ on some interval $[0,2a],\ a>0$. This means by definition,  that  $g$ admits a representation
\[
g(t)=-\eta|t|-\int_0^t(t-s)h(s)\,\dd s,\quad t\in[0,2a],
\]
with some $\eta >0$ and a continuous function $h$ on $[0,2a]$. In particular, $g$ is twice continuously differentiable on $(0,2a)$. Then on $[0,a]$ the corresponding canonical system can be written as a Dirac-Krein system, that is in the form \eqref{pot1} with a continuous potential
\[
V(x)=\begin{pmatrix}\be(x)&-\al(x)\\-\al(x)&\-\be(x)\end{pmatrix},\quad x\in[0,a],
\]
see e.g. \cite{KL4}. According to the above, in this case statement ${4.4^{\rm o}}$ applies.
\end{remark}

\medskip

\noindent{4.4.}\   Recall (see \cite{KK1}) that a {\it string} $S[\ell,M]$ is given by its length~$\ell$, $0<\!\ell\le~\!\!\infty$, and its mass distribution $M$ on $[0,\ell)$, that is, $M(x)$ is the mass of the interval $[0,x],\,0\le x\le\ell$, and we set $M(x)=0$ if $x<0$. Then $M$ is a non-decreasing function on $(-\infty,\ell)$. We always suppose that $M(x)>0$ if $x>0$. The equation 
\begin{equation}\label{81}
\dd y'(x)+z y(x)\,\dd M(x)=0,\quad 0\le x<\ell,\quad z\in\mathbb C,
\end{equation}
 is called the {\it differential equation of the string} $S[\ell,M]$. This string  is called {\it regular} if its length and its total mass are finite: $\ell+M(\ell)<\infty$; otherwise it is called {\it singular}. If the string is regular we assume that $M(\ell-0)=M(\ell)$.

We introduce the solutions $\varphi,\,\psi$
of  equation \eqref{81} that satisfy the initial conditions
\begin{equation*}%\label{82}
\varphi(0;z)=1,\ \varphi'(0-;z)=0;\quad \psi(0,z)=0,\ \psi'(0-;z)=1.
\end{equation*}
That is, $\varphi(x,z),\,\psi(x;z)$ are the solutions of the integral equations
\begin{equation}\nonumber%label{83}
\begin{array}{rcl}
\varphi(x;z)&=&1+z\displaystyle\int_{0-}^x\,(x-s)\varphi(s;z)\,\dd M(s),\\[3mm]
\psi(x;z)&=&x+z\displaystyle\int_{0-}^x\,(x-s)\psi(s;z)\,\dd M(s).
\end{array}
\end{equation}
%From these equations it follows that in each point $x\in(0,\ell)$ there exist left and right derivatives $\varphi'_\pm,\,\psi'_\pm$, that 
%\[
%\varphi_\pm'(x;\la)=\varphi'(x\pm;\la),\quad \psi_\pm'(x;\la)=\psi'(x\pm;\la),\quad x\in[0,\ell),
%\]
%and that the left and right derivatives coincide in points of continuity of $M$. Moreover, 
%
%Similar for a dtring: $M_1=M_2$ on $[0,a] \Longleftrightarrow h_1(t)=h_2(t)$ on $[0,2\int_0^a\sqrt{M'(x)}\dd x]$.
 The set of all {\it spectral measures} $\tau$ of the regular string $S[\ell,M]$ can be defined by the relation 
\begin{equation}\nonumber%label{811}
\dfrac{\psi'(\ell;z)\ga(z)+\psi(\ell;z)}{\varphi'(\ell;z)\ga(z)+\varphi(\ell;z)}=\int_0^\infty\,\dfrac{\dd \tau(\la)}{\la-z},
\end{equation}
if $\ga$ runs through the class  $\boldS$ of all {\it Stieltjes functions};  recall that by definition $\ga\in\boldS$ if $\ga$ is holomorphic in $\CC\setminus [0,\infty)$ and $\ga,\wh\ga\in\boldN$, where $\wh\ga(z):=z\ga(z)$.
The {\it transfer function} corresponding to the spectral measure $\tau$ is  the function
\[
 g_\tau(t)=\int_0^\infty\dfrac{\cos\big(\sqrt{\la}\, t\big)-1}{\la}\,\dd\tau(\la),\quad z\in\RR.
\] 
Now  analogs of the statements ${4.1^{\rm o}}$ and ${4.2^{\rm o}}$ hold with $\det H(x)$ replaced by $M'(x)$; here $M'$ denotes the derivative of the absolutely continuous  component of the non-decreasing function $M$. For details the reader is referred to \cite{KL5}. We only formulate the analogue of ${4.3^{\rm o}}$ for  regular strings.
\begin{itemize} {\it
\item[${4.5^{\rm o}}\!\!.$] Let $S[\ell_j,M_j]$ be a regular string such that $a_j:=\int_0^{\ell_j}\,\sqrt{M'_1(x)}\,\dd x>0$, and with transfer function  $g_j$,   $j=1,2$. If $0<a<\min \{a_1,a_2\}$ and
\[
l(a):=\inf\left\{l:\int_0^{l}\sqrt{M'_1(x)}\,\dd x=a\right\},
\]
then 
\[
g_1\big|_{[0,2a]}=g_2\big|_{[0,2a]}\ \Longleftrightarrow\ M_1\big|_{[0,l(a)]}= M_2\big|_{[0,l(a)]}.
\]

}
\end{itemize}

\medskip

If the string $S[\ell,M]$ has a concentrated mass at $x=0\!:\,M(0)>0$, then the spectral measures $\tau$ of the string are finite and in ${4.5^{\rm o}}$ the transfer function  $g_\tau$ can be replaced by 
\begin{equation}\nonumber%label{ef}
 f_\tau(t)=\int_0^\infty\cos\big(\sqrt{\la}\, t\big)\dd\tau(\la),\quad z\in\RR.
\end{equation}

The characteristic property of a continuous function $g$ (or $f$) to have a representation
\[
g(t)=\!\!\int_0^\infty\!\dfrac{\cos\big(\sqrt{\la}\, t\big)-1}{\la}\,\dd\tau(\la)\quad \!\left(\!\text{ or } f(t)=\int_0^\infty\!\cos\big(\sqrt{\la}\, t\big)\dd\tau(\la)\!\right),\quad t\in\RR,
\]
with a measure $\tau$ such that $\int_0^\infty\frac{\dd\tau(\la)}{1+\la}<\infty$ $\big($or $\int_0^\infty\dd\tau(\la)<\infty\big)$ is that  the kernel $G_g$ from \eqref{kerG} is positive definite and $g$ is real (or  the kernel  $F_f$ from \eqref{efef} is positive definite and $f$ is real).

\section{Appendix}\label{app}
\noindent{5.1.} Proof of statement $2.2^{\rm o}$. The solutions $\varphi(x;\la)$ and the functions $\cos(\sqrt{\la}x)$ in Section 2 are connected by Volterra integral equations
\begin{align}\label{lev1}
\varphi(x,\la)&=\cos(\sqrt{\la}x)+\int_0^xK(x,\xi)\cos(\sqrt{\la}t)\dd \xi,\quad 0\le x<\ell,\\\label{lev2}
\cos(\sqrt{\la}x)&=\varphi(x,\la)-\int_0^xK_1(x,\xi)\varphi(\xi,\la)\dd \xi\quad 0\le x<\ell,
\end{align}
with kernels $K(x,\xi),\,K_1(x,\xi)$, see \cite{LG}, and also  \cite[Section IV.11]{L-shift}, \cite{M}, \cite{FY};  if $q$ has $m$ locally summable derivatives then $K$ has in both variables $m+1$ locally summable derivatives. Integrating \eqref{lev2} with respect to $x$ from $0$ to $a<\ell$ we find
\[
\frac{\sin(\sqrt{\la}a)}{\sqrt{\la}}=\int_0^a\varphi(x,\la)\left[1-\int_x^aK_1(\xi,x)\dd\xi\right]\dd x.
\]
 Hence the function $\frac{\sin(\sqrt{\la}a)}{\sqrt{\la}}$ is the Fourier transformation of the function
\[
K_2(a,x)=\left\{\begin{array}{cl}1-\displaystyle\int_a^x K_1(\xi,x)\dd\xi,\quad&0\le x\le a,\\[2mm]
0&x>a,
\end{array}\right.
\]
 Parseval's relation implies for an arbitrary $\tau\in\mathcal S_\ell$
\[
2\int_\mathbb R \left(\frac{\sin(\sqrt{\la}a)}{\sqrt{\la}}\right)^2\dd \tau( \la)=2\int_0^aK_2(a,x)^2\dd x,
\]
or
\[
\int_\mathbb R\frac{1-\cos(2\sqrt{\la}a)}{\la}\dd\tau(\la)=2\int_0^aK_2(a,x)^2\dd x.
\]
Since the integral on the right hand side is finite for all $a<\ell$ and independent of $\tau\in\mathcal S_\ell$ ,  the function $\Phi_\tau(t)$ is well defined and independent of $\tau\in\mathcal S_\ell$ for $0\le t<2\ell$. This relation does also imply that $\Phi_\tau(t)$ for $0\le t\le 2a$ is independent of $\tau\in\mathcal S_a$, and statement ${2.2^{\rm o}}$ is proved.

\medskip

\noindent{5.2.} In this subsection we outline the application of the method of directing functionals (see \cite{Knuclei}, \cite{KUkr}, \cite{Lrf}) to the kernel $K_\Phi$ and indicate a proof of statement ${2.3^{\rm o}}$.

Consider a continuous function $\Phi$ on $(0,2\ell)$ with the property that the kernel
\[
K_\Phi(s,t)=\frac 12\big(\Phi(t+s)-\Phi(|t-s|)\big),\quad 0\le s,t<\ell,
\]
 is positive definite. By $\mathcal L_\Phi$ we denote the Hilbert space which is obtained if the space $C_0([0,\ell))$ of continuous functions on $[0,\ell)$, which vanish identically near $\ell$, is equipped with the inner product
\begin{equation}\label{ip}
[u,v]_\Phi:=\int_0^\ell\int_0^\ell K_\Phi(s,t)u(s)\ov{v(t)}\dd s\dd t,\quad u,v\in C_0([0,\ell)),
\end{equation} 
and  factored and completed in a canonical way.

The operator
\begin{equation}\label{ww}
 B_0:\ B_0u:=\dfrac{\dd ^2u(t)}{\dd t^2},\ t\in[0,\ell),
\end{equation}
 where 
\[
u\in{\rm{dom}}\,B_0:=\{u\in C^2([0,\ell)), u(0)=0, \,u \text{ vanishes identically near } \ell\},
\] 
is symmetric with respect to the inner product \eqref{ip} and hence  generates a closed symmetric operator $B$ in $\mathcal L_\Phi$. A directing functional of $B_0$ is
\[
 \mathcal G(u;\la):=\displaystyle\int_0^\ell\,u(t)\dfrac{\sin(\sqrt{\la}t)}{\sqrt{\la}}\dd t,\quad u\in C_0([0,\ell)).
\]
 Having one directing functional, the defect numbers of $B$ are zero or one, and it is easy to see that they are equal.
The method of directing functionals yields the existence   of a unique or of infinitely many {\it spectral measures $\tau$ of the operator $A$}. This means that for $\tau$ there holds Parseval's relation
\begin{equation}\label{parrf}
[u,u]_\Phi=\int_0^\ell|\mathcal G(u;\la)|^2\dd \tau(\la),\quad u\in\mathcal L_\Phi,
\end{equation}
which implies 
\begin{align}\label{end}
\frac 12\big(\Phi(t\!+\!s)-\Phi(|t\!-\!s|)\big)&=\int_\RR\frac{\sin(\sqrt{\la}t)}{\sqrt{\la}}\frac{\sin(\sqrt{\la}s)}{\sqrt{\la}}\dd\tau(\la),\quad 0\le s,t<\ell.
\end{align}
The relation \eqref{parrf} means that the directing functional $\mathcal G(u;\cdot)$ defines an isometry from $\mathcal L_\Phi$
into $L^2_\tau$. The spectral measure $\tau$ is called {\it orthogonal} if this isometry is onto, and otherwise {\it non-orthogonal}.The set of all spectral measures of $B$ is denoted by $\mathcal T_\ell$, and $\mathcal T_\ell^{\rm orth}$ denotes its subset of orthogonal spectral measures.

It follows from \eqref{end} that
\begin{align*}
\Phi(2t)-\Phi(0)&=2\int_\RR\frac{(\sin(\sqrt{\la}t))^2}\la\,\dd\tau(\la)\\
&= \int_\RR \frac{1-\cos(2\sqrt{\la}t)}{\la}\,\dd\tau(\la),\quad 0\le t<\ell,
\end{align*}
and, since $\Phi(0)=0$,
\begin{equation}\label{stock}
\Phi(t)=\int_\mathbb R\dfrac{1-\cos(\sqrt{\la}t)}{\la}\dd \tau(\la),\quad 0\le t<2\ell.
\end{equation}

In the method of directing functionals it is shown that the set of all spec\-tral measures of $B$ is in a bijective correspondence with the set of all self-adjoint extensions of $B$. In fact, a spectral measure is  orthogonal if the corresponding self-adjoint extension of $B$ acts in $\mathcal L_\Phi$, and it is non-orthogonal  if the corresponding self-adjoint extension  acts in a properly larger space than~$\mathcal L_\Phi$.

Hence there is a bijective correspondence between all self-adjoint extensions of $B$ in $\mathcal L_\Phi$ or in a larger Hilbert space, and all  representations of $\Phi$ in the form \eqref{stock}. 

Now let $0<a<\ell$ and consider the corresponding sets $\mathcal T_a$ and $\mathcal T_a^{\rm orth}$  for the restriction of $\Phi$ to $[0,2a]$. Clearly, from these definitions, 
\[
\mathcal S_a\subset \mathcal T_a.
\]
 and statement $2.3^{\rm o}$ can be formulated as follows:
\begin{equation}\label{st}
 \mathcal S_a=\mathcal T_a.
\end{equation}

 We prove the corresponding relation  for the orthogonal spectral measures:
\begin{equation}\label{stock1}
\mathcal S_a^{\rm orth}=\mathcal T_a^{\rm orth}.
\end{equation}
To this end we first show that
\begin{equation}
\mathcal S_a^{\rm orth}\subset\mathcal T_a^{\rm orth}.
\end{equation}
%(2) if $\tau\in\mathcal T_a^{\rm orth}$ and $\tau\in\mathcal S_a$ then $\tau\in\mathcal S_a^{\rm orth}$.

Consider $\tau\in\mathcal S_a^{\rm orth}$. If $\tau$ would not be in $\mathcal T_a^{\rm orth}$ there would exist an $h\in L^2_\tau, h\ne 0$, such that $h\perp \dfrac{\sin(\sqrt{\la}x)}{\sqrt{\la}},\ x\in [0,a)$. Now we observe, that the relation \eqref{lev1}
%\begin{equation}\label{lev}
%\varphi(x,\la)=\cos(\sqrt{\la}x)+\int_0^xK(x,t)\cos(\sqrt{\la}t)\dd t
%\end{equation}
%which holds with some function $K(s,t)$, see (Slava!) Levitan/Gasymov, UMN 19,2 (116), 1964,3-63 (see also Levitan, Inverse SL problems, 1985,  Levitan/Sargsjan, 1988, VI.3, and Levitan, Operators of the generalized shift ...  1962, IV.11);  if $q$ has $m$ locally summable derivatives then $K$ has in both variables $m+1$ locally summable derivatives. Relation \eqref{lev} 
implies
\begin{align*}
\int_0^x&\varphi(\xi,\la)\dd \xi =\dfrac{\sin(\sqrt{\la}x)}{\sqrt{\la}}+\int_0^x\int_0^\xi K(\xi,t)\cos(\sqrt{\la}t)\dd t\dd \xi\\
=&\dfrac{\sin(\sqrt{\la}x)}{\sqrt{\la}}+\int_0^x\left(K(t,t)-\int_t^xK_t(\xi,t)\dd \xi\right)\dfrac{\sin(\sqrt{\la}t)}{\sqrt{\la}}\dd t,\quad 0\le x\le a.
\end{align*}
It follows that also $h\perp \int_\Delta \varphi(x;\la)\dd x$ for all intervals $\Delta\subset[0,\ell)$ in $L^2_\tau$, hence $\tau\notin\mathcal S_a^{\rm orth}$, a contradiction.

 %The claim (2) follows in the same way, since the relation
%\[
%\cos(\sqrt{\la}x)=\varphi(x,\la)+\int_0^x L(x,t)\varphi(t,\la)\dd t
%\]
%implies
%\[
%\dfrac{\sin(\sqrt{\la}x)}{\sqrt{\la}}=\int_0^x \varphi(\xi,\la)\dd \xi+\int_0^x\int_0^\xi L(\xi,t)\varphi(t,\la)\dd t\dd \xi
%\]

As is well-known (it follows e.g. from M.G.~Krein's resolvent formula), for any given real number $\la $, there is exactly one orthogonal spectral measure in $\mathcal S_a$ which has $\la$ in its support, and the same holds for $\mathcal T_a$. Therefore the two sets of orthogonal spectral measures coincide and \eqref{stock1} is proved. 
%The corresponding Nevanlinna functions
%\[
%\int_\mathbb R\,\dfrac{\dd\tau(\la)}{\la-z},\quad z\in\mathbb C^+\cup\mathbb C^-
%\] 
%are given by fractional linear transformations with parameters $\gamma\in\wt{\mathbb R}$. ???Then the non-canonical spectral measures are obtained from these fractional linear transformations inserting $\gamma\in\wt{\bf N}$.???

For a proof of \eqref{st} we observe that according to \eqref{bloch} the set $\mathcal S_a$ is given through a fractional linear relation 
\begin{equation}\label{f1}
\int_\mathbb R\frac{\dd\tau(\la)}{\la-z}=\frac{a_{11}(z)\ga(z)+a_{12}(z)}{a_{21}(z)\ga(z)+a_{22}(z)};
\end{equation} 
for the right hand side we write for short $W_\mathcal A(\gamma(z)),\ \mathcal A(z):=(a_{ij})_{i,j=1}^2$. The theory of resolvent matrices (see, e.g. \cite{KLII}) yields a similar representation for the set $\mathcal T_a$:
\begin{equation}\label{f2}
\int_\mathbb R\frac{\dd\tau(\la)}{\la-z}=\frac{b_{11}(z)\ga(z)+b_{12}(z)}{b_{21}(z)\ga(z)+b_{22}(z)}=W_{\mathcal B(z)}(\gamma(z)).
\end{equation} 
The relation \eqref{stock1} implies that for each $\ga\in\mathbb R\cup\{\infty\}$ there exists a $\wh\ga\in\mathbb R\cup\{\infty\}$ such that $\mathcal W_{\mathcal  A(z)}(\ga)=\mathcal W_{\mathcal  B(z)}(\wh\ga)$, and this mapping is a bijection in $\mathbb R\cup\{\infty\}$:
\[
\mathcal W_{\mathcal  B(z)^{-1}}\big(\mathcal W_{\mathcal  A(z)}(\ga)\big)=\mathcal W_{\mathcal  B(z)^{-1}\mathcal  A(z)}(\ga)=\wh\ga.
\]
Hence $\mathcal  B(z)^{-1}\mathcal  A(z)=a(z)C$ with some scalar function $a(z)$ and a constant $J$-unitary matrix C, and \eqref{stock1} follows easily.

\medskip
 
\noindent{5.3.} If $\wt A$ denotes the self-adjoint extension of $A$ which corresponds to $\tau$, the left hand side in \eqref{f1} can be written, at least formally, as $((\wt A-z)^{-1}\delta_0,\delta_0)$; similarly, if $\wt B$ denotes the self-adjoint extension of $B$ corresponding to $\tau$  the left hand side in \eqref{f2} becomes $[(\wt B-z)^{-1}\delta'_0,\delta'_0]_\Phi$. Here $\delta_0$ and $\delta_0'$ are to be considered as generalized elements of $L^2(0,a)$ and $\mathcal L_\Phi$, respectively.    

A consequence of the relation \eqref{stock1} is the following statement: The Sturm-Liouville operator $\wt A$ in $L^2(0,a)$, given by \eqref{hc} with constant $\ga$, is unitarily equivalent to a self-adjoint extension $\wt B$ of the closure $B$ of $B_0$ from \eqref{ww} in $\mathcal L_\Phi(0,a)$, in fact, both operators are unitarily equivalent to the operator of multiplication by the independent variable $\la$ in $L^2_\tau(\mathbb R)$:
\[
\begin{array}{ccccc}
\wt A \text{ in } L^2(0,a)\!\!&\!\!\xlongrightarrow[]{\int_0^ay(x)\varphi(x;\la)\dd x}\!\!&\!\!\la\!\cdot\!\! \text{ in } L_\tau^2(\mathbb R)\!\!&\!\!\xlongleftarrow[]{\int_0^a u(t)\frac{\sin(\sqrt{\la}t)}{\sqrt{\la}}\dd t}\!\!&\!\!\wt B \text{ in } \mathcal L_\Phi(0,a);
\end{array}
\]
The unitary equivalence is realized through the two Fourier transformations or their inverses.
Here $\wt A$ is the operator given by 
\[
-\dfrac{\dd^2y}{\dd x^2}+qy,\quad
 y'(0)-   h\, y(0)=0,\  y'(a)-\ga y(a)=0,
\]
 and $\wt B$ is a self-adjoint extension of 
\[
\frac{\dd^2u}{\dd t^2},\quad u(0)=0,\ u \text{ vanishing near }a. 
\]
This is in analogy to the Hamburger moment problem, where $\wt A$ corresponds to the operator generated by the Jacobi matrix
 in $l^2(\mathbb N_0)$ and $\wt B$ to the shift operator in the Hilbert space generated by the positive definite kernel  $k(m,n):=s_{m+n},\ m,n=0,1,\dots$, where $(s_n)_0^\infty$ is the moment sequence. 

A corresponding remark holds also for the transfer functions of the canonical systems and strings in Section 4.

%\end{document }

\end{document}